\documentclass[11pt,reqno]{amsart}
\usepackage{graphicx,amsmath,amssymb,amsfonts,amsthm,enumitem,bm,xcolor}
\usepackage{a4wide,fullpage}
\usepackage{cleveref}
\usepackage{float}
\usepackage[normalem]{ulem}

\usepackage{url}

\usepackage{color}

\usepackage{xurl}

\newcommand{\N}{\mathbb{N}}
\newcommand{\R}{\mathbb{R}}

\renewcommand{\a}{\alpha}
\renewcommand{\b}{\beta}

\renewcommand{\d}{{\delta}}
\newcommand{\g}{\gamma}
\newcommand{\G}{\Gamma}
\renewcommand{\l}{\lambda}

\renewcommand{\k}{\kappa}

\renewcommand{\t}{\tau}

\newcommand{\z}{\zeta}

\renewcommand{\(}{\left\(}
\renewcommand{\)}{\right\)}
\renewcommand{\i}{\infty}

\newcommand{\pa}[2]{\left(\frac{#1}{#2}\right)}

\numberwithin{equation}{section}
\theoremstyle{plain}
\newtheorem{theorem}{Theorem}[section]
\newtheorem{lemma}[theorem]{Lemma}

\newtheorem{corollary}[theorem]{Corollary}
\newtheorem{proposition}[theorem]{Proposition}

\theoremstyle{definition}

\numberwithin{equation}{section}

\crefname{subsection}{subsection}{Subsection}

\renewcommand{\pmod}[1]{\ \left( \mathrm{mod} \, #1 \right)}

\setlist[enumerate]{leftmargin=*,label=\rm{(\arabic*)}}

\makeatletter
\@namedef{subjclassname@2020}{%
	\textup{2020} Mathematics Subject Classification}
\makeatother

\allowdisplaybreaks

\title{On sign patterns of Fourier coefficients of theta quotients}
\author{Koustav Banerjee}
\author{Kathrin Bringmann}
\author{Bernhard Heim}
\address{University of Cologne, Department of Mathematics and Computer Science, Weyertal 86-90, 50931 Cologne, Germany}
\email{kbanerj1@uni-koeln.de}
\email{kbringma@math.uni-koeln.de}
\email{bheim@uni-koeln.de}
\author{Ben Kane}
\address{The University of Hong Kong, Department of Mathematics, Pokfulam, Hong Kong}
\email{bkane@hku.hk}

\makeatletter
\@namedef{subjclassname@2020}{%
	\textup{2020} Mathematics Subject Classification}
\makeatother

\subjclass[2020]{{\hspace{-0.08 cm}11F03, 11F11, 11F20, 11F27, 11F30, 11F37.}}

\begin{document}
\begin{abstract}
 Motivated by conjectures of Schlosser and Zhou on sign patterns of Fourier coefficients of infinite products, we study two families of theta quotients and their positive integral powers. We determine asymptotic formulas for their Fourier coefficients and use them to obtain explicit eventual periodic sign patterns. Our results provide new instances of the sign phenomena observed by Schlosser and Zhou and illustrate how dominant cusps determine the resulting sign patterns. These results were formalized and verified in Lean by AxiomProver.
\end{abstract}
\maketitle

\section{Introduction and statement of results}

In this paper, for $m\in \mathbb{N}_{\ge 3}$ and $\ell\in \mathbb{Z}\setminus\{0\}$, we study the sign pattern of the Fourier coefficients for the quotient
\begin{equation}\label{prod}
\frac{\left(q,q^{m-1};q^m\right)^{\ell}_{\infty}}{\left(q^2,q^{m-2};q^m\right)^{\ell}_{\infty}}.
\end{equation}
Here and throughout, $|q|<1$ and for $a\in\mathbb{C}$ and $n\in\mathbb{N}_0\cup \{\infty\}$, we set $(a)_n\hspace{-0.1 cm}=\hspace{-0.1 cm}(a;q)_n\hspace{-0.1 cm}:=\prod_{j=0}^{n-1}(1-aq^j)$. For $r\in \mathbb{N}$ and $a_1,\dots,a_r\in \mathbb{C}$, we let $(a_1,\dots,a_r)_n:=(a_1,\dots,a_r;q)_n:=(a_1)_n\dots (a_r)_n$.  The primary motivation for considering this problem goes back to the sign pattern of Fourier coefficients of the Rogers--Ramanujan continued fraction 
\[R(q):=\frac{q^{\frac15}}{1+\dfrac{q}{1+\dfrac{q^{2}}{1+\dfrac{q^{3}}{1+\cdots}}}}=q^{\frac15}\frac{(q,q^{4};q^{5})_{\infty}}{(q^{2},q^{3};q^{5})_{\infty}}.\]
In this case, using the Circle method, the sign pattern of the $n$-th Fourier coefficient for $n$ sufficiently large 
\rm
was proved by Richmond and Szekeres \cite[(3.9)]{RichmondSzekeres1978}. Andrews \cite[Corollary 1]{Andrews1981} then gave a proof
\rm
for the sign pattern of all coefficients
\rm
via $q$-series dissections. Later, using dissections of Ramanujan's theta functions, Ramanathan \cite[Theorem 5 and Theorem 6]{Ramanathan} and Chan--Yesilyurt \cite[Theorem 1.2 and Theorem 1.3]{CY} studied periodicity of the sign of Fourier coefficients for a family of infinite products resembling
\rm
\eqref{prod}. Recently, using similar techniques,
\rm
Baruah and Sarma \cite[Theorem 2 and Theorem 3]{BS} investigated the sign pattern of the coefficients of \eqref{prod} with $m=5$ and $\ell=\pm 5$. Using the Circle method,  Schlosser and Zhou \cite{SchlosserZhou2023} studied the sign pattern of Fourier coefficients of theta quotients similar to \eqref{prod} and made several conjectures. Three of the authors \cite[Theorem 1.2]{BHK} settled one of the conjectures of Schlosser and Zhou \cite[Conjecture 22]{SchlosserZhou2023}. Other results in this direction can be found in \cite[Section 2]{Chern} and \cite{HeLi2025}.

Returning to \eqref{prod},
\rm
for $m\in \mathbb{N}_{\ge 3}$ and $\ell\in \mathbb{N}$ we write
\rm
\begin{equation}\label{Eqn0}
F^{\ell}_m(q):=\frac{\left(q,q^{m-1};q^m\right)^{\ell}_{\infty}}{\left(q^2,q^{m-2};q^m\right)^{\ell}_{\infty}}=:\sum_{n\ge 0}c^{[\ell]}_m(n)q^n
\end{equation}
and study the asymptotic growth and sign pattern of $c^{[\ell]}_m(n)$ as $n\to\infty$. 

We first consider the case $m$ odd. For $m=3$, we trivially have $F_m^{\ell}(q)=1$ and hence $c_{m}^{[\ell]}(n)=0$ for all $n\in\N$. For $m\in \mathbb{N}_{>3}$ odd, define
\begin{equation}\label{phase}
\phi_m:=\frac 12\left(1-\frac 2m\right)\left(1-\frac 3m\right).
\end{equation}
Our first result is the following: 
\begin{theorem}\label{A}
	Let $m\in \mathbb{N}_{>3}$ be odd and $\ell\in \mathbb{N}$. Then, as $n\to \infty$,
	\begin{align*}
	c^{[\ell]}_m(n)
	&\sim
	\left(
	\frac{\left(1-\frac{1}{m}\right)\left(1-\frac{3}{m}\right)\ell}{2m}
	\right)^{\frac14}
	\frac{
		(-1)^n
		\cos\left(\pi\left(\phi_m\ell-\frac{n}{m}\right)\right)
		e^{\pi\sqrt{\frac{2\left(1-\frac{1}{m}\right)\left(1-\frac{3}{m}\right)\ell}{m}n}}
	}{
		n^{\frac34}
	}.
	\end{align*}
	In particular, for $n\gg1$,
	\[
	\textnormal{sgn}\left(c^{[\ell]}_m(n)\right)= (-1)^n\textnormal{sgn}\left(\cos\left(\pi\left(\phi_m\ell-\frac nm\right)\right)\right).
	\]
\end{theorem}

Next, we  explicitly determine the sign in Theorem \ref{A}. For a property $P$, define
\[
\chi_P:=\begin{cases}
1 &\text{if}\ \text{$P$ holds},\\
0&\text{otherwise}.
\end{cases}
\] 

Moreover let
\[
\mathcal{S}^{[\ell]}_{m,1}:=\begin{cases}
\left[\frac{m+1}{2}-\frac{5\ell}{2},\frac{3m-1}{2}-\frac{5\ell}{2} \right]&\ \text{if}\ \ell\ \text{is even},\\[0.3 em]
\left[m-\frac{5\ell-1}{2},2m-\frac{5\ell+1}{2} \right]&\ \text{if}\ \ell\ \text{is odd}.
\end{cases}
\]

\begin{corollary}\label{BenCor}
	Let $\ell\in \mathbb{N}$, $m\in \mathbb{N}_{>3}$ odd with $m>6\ell$, and define $0\le r\le 2m-1$ via $n\equiv r\pmod{2m}$. Then, for $n\gg1$, 
	\[
	\textnormal{sgn}\left(c^{[\ell]}_{m}(n)\right)=(-1)^{r+\left\lfloor\frac{\ell}{2} \right\rfloor+\chi_{r\in\mathcal{S}^{[\ell]}_{m,1} } }.
	\]
\end{corollary}

Next, we consider the growth of $c^{[\ell]}_m(n)$
\rm
for $m\in \mathbb{N}_{\ge 4}$ even and $\ell\in \mathbb{N}$.

\begin{theorem}\label{mainthmeven}
	We have, for $m\in \mathbb{N}_{\ge 4}$ even and $\ell\in \mathbb{N}$, as $n\to \infty$, 
	\[
	c^{[\ell]}_m(n)\sim \frac{\ell^{\frac 14}}{2^{\ell+\frac 54}m^{\frac 14}\sin^{\ell}\left(\frac{2\pi}{m}\right)}\frac{(-1)^ne^{\pi\sqrt{\frac{2\ell}{m}n}}}{n^{\frac 34}}.
	\]
\end{theorem}

In this case, we have the following sign-pattern of $c^{[\ell]}_m(n)$.
\begin{corollary}\label{mainthmevencor}
	We have, for $m\in \mathbb{N}_{\ge 4}$ even, $\ell\in \mathbb{N}$, and $n\in \mathbb{N}_0$,
	\[
	\textnormal{sgn}\left(c^{\left[\ell\right]}_{m}(n)\right)=(-1)^n.
	\]
\end{corollary}

We next turn to $\ell<0$. Writing $\ell=-\l$ with $\l\in\mathbb{N}$, for $m\in \mathbb{N}_{>3}$ with $\gcd(m,6)=1$, set
\begin{equation*}
\mathcal{F}^{\l}_m(q):=\frac{\left(q^2,q^{m-2};q^m\right)^{\l}_{\infty}}{\left(q,q^{m-1};q^m\right)^{\l}_{\infty}}=:\sum_{n\ge 0}\g^{[\l]}_m(n)q^n.
\end{equation*}

If $m\equiv \pm 1 \pmod{6}$ and $\l\in \mathbb{N}$, define
\begin{align*}
T^{[\l]}_m(n)&:=\cos\left(\pi\left( \frac{\l}{6}-\frac{ n}{3}\pm\left(\frac{ 2n-\l}{6m}+\frac{\l}{m^2}\right)\right)\right).
\end{align*}

We show the following result:  
\begin{theorem}\label{A1}
	Let $m\in \mathbb{N}_{>3}$ with $\gcd(m,6)=1$ and $\l\in \mathbb{N}$. Then, as $n\to \infty$, 
	\begin{align*}
	\g^{[\l]}_m(n)&\sim \left(\frac{\left(1-\frac 1m\right)\left(1+\frac 1m\right)\l}{6m}\right)^{\frac 14}\frac{T^{[\l]}_m(n)e^{\pi\sqrt{\frac{2\left(1-\frac 1m\right)\left(1+\frac 1m\right)\l}{3m}n }}}{n^{\frac 34}}.
	\end{align*}
\end{theorem}

To state the sign pattern of $\g^{[\l]}_m(n)$ for $m\equiv \pm 1\pmod{6}$, we define
\[
\mathcal{S}^{[\l],\pm}_{m,2}:=\begin{cases}
\left[
\frac{m+1}{2}+\frac{\l}{6},\,
\frac{3m-1}{2}+\frac{\l}{6}
\right] &\text{if}\ \ \l\equiv0\pmod{6},\\[0.3 em]
\left[
\frac{(3\mp1)m+4}{6}+\frac{\l-1}{6},\,
\frac{(9\mp1)m-2}{6}+\frac{\l-1}{6}
\right]       &\text{if}\ \ \l\equiv1\pmod{6},\\[0.5 em]
\left[
\frac{(3\mp2)m+5}{6}+\frac{\l-2}{6},\,
\frac{(9\mp2)m-1}{6}+\frac{\l-2}{6}
\right]        &\text{if}\ \ \l\equiv2\pmod{6},\\[0.3 em]
\left[
\frac{\l+3}{6},\,
m+\frac{\l-3}{6}
\right]       &\text{if}\ \ \l\equiv3\pmod{6},\\[0.3 em]
\left[
\frac{(3\pm2)m+7}{6}+\frac{\l-4}{6},\,
\frac{(9\pm2)m+1}{6}+\frac{\l-4}{6}
\right]          &\text{if}\ \ \l\equiv4\pmod{6},\\[0.5 em]
\left[
\frac{(3\pm1)m+8}{6}+\frac{\l-5}{6},\,
\frac{(9\pm1)m+2}{6}+\frac{\l-5}{6}
\right]        &\text{if}\ \ \l\equiv5\pmod{6}.
\end{cases}
\]

\begin{corollary}\label{FlipCor}
	Let $\l\in\mathbb N$, $m\in\mathbb N_{>3}$ with $m\equiv \pm 1\pmod{6}$ and
	$\l < \frac{m}{2}$, and assume that $n\equiv 3r+2mj\pmod{6m}$ with $0\le r\le2m-1$ and $0\le j\le2$. Then, for $n\gg1$, 
	\[
	\operatorname{sgn}\left(\g^{[\l]}_m(n)\right)=(-1)^{r+\left\lfloor\frac{\l}{6}\right \rfloor+\chi_{r\in \mathcal{S}^{[\l],\pm}_{m,2}}}\begin{cases}
	1 &\text{if}\ \ \l\equiv 0,1,2\pmod{6},\\
	-1 &\text{if}\ \ \l\equiv 4,5\pmod{6},\\
	\left(\frac m3\right)&\text{if}\ \ \l\equiv 3\pmod{6},
	\end{cases}
	\]
	where $\left(\frac{\cdot}{\cdot}\right)$ denotes the Jacobi-symbol.
\end{corollary}

The paper is organized as follows. In \Cref{sec1}, we state an exact formula of Zuckerman, preliminary results on the Jacobi theta function, and Ingham's Tauberian theorem. Proofs of \Cref{A} and \Cref{BenCor} are given in \Cref{sec2}. We show \Cref{mainthmeven} and \Cref{mainthmevencor} in \Cref{sec3}. \Cref{sec4} is devoted to the proof of \Cref{A1} and \Cref{FlipCor}. Finally, we conclude this paper with \Cref{sec5} by proposing questions for future research.

\section*{Acknowledgements}
 The first and the second author have received funding from the European Research Council (ERC) under the European Union’s Horizon 2020 research and innovation programme (grant agreement No. 101001179). The research of the fourth author was supported by grants from the Research Grants Council of the Hong Kong SAR, China (project numbers HKU 17314122 and HKU 17305923). The authors are grateful to Ken Ono for using AxiomProver to formalize the results of this paper and to verify the proofs in Lean. These formal verifications were carried out by AxiomProver. The relevant files are available in the following repository:
 \url{https://github.com/bkanemath/Sign-Patterns-of-Fourier-Coefficients-of-Theta-Quotients}.

\section{Preliminaries}\label{sec1}

\subsection{Exact formulas of Rademacher and Zuckerman}

We require an exact formula for Fourier coefficients of weakly holomorphic modular functions. Rademacher \cite{Rademacher} perfected the Circle Method of Hardy and Ramanujan \cite{HardyRamanujan} and  obtained an exact formula for the partition function. This was later extended to weakly holomorphic forms of nonpositive weight in results of Rademacher--Zuckerman \cite{RademacherZuckerman} and Zuckerman \cite{Zuckerman} and by Ono and the second author \cite{BringmannOno}. To state the result, let $I_\alpha$ denote the usual $I$-Bessel function. We allow weakly holomorphic modular forms with a multiplier system. Throughout, for $a\in \mathbb{N}$, $b\in \mathbb{Z}$, with $\gcd(a,b)=1$, we denote by $[b]_a$ the inverse of $b\pmod a$.
\begin{theorem}\label{Thm:Zuckerman}
	Let $\kappa\leq 0$. Suppose that  $f$ is a weakly holomorphic modular form of weight $\k$ on a subgroup $\G$ of finite index in ${\rm SL}_2(\mathbb{Z})$. Assume that, for some $\d\ge 0$, $f(\tau) = \sum_{n\gg-\infty} a(n)q^{n+\delta}$. For $h,k\in \mathbb{N}_0$ with $0\le h<k$ and $\gcd(h,k)=1$, define $f_{h,k}$ by 
	\[
	f\left(\frac 1k\left(h+iz\right)\right)=\pa{i}{z}^{\k}f_{h,k}\left(\frac 1k\left(\left[-h\right]_k+\frac{i}{z}\right)\right)\ \ \left({\rm Re}(z)>0\right).
	\] 
Write $f_{h,k}(\t)=:\sum_{r\gg-\infty}a_{h,k}(r)q^{\frac{r+\d_{h,k}}{c_k}}$. Here $c_k\in \mathbb{N}$ and $\d_{h,k}\in [0,1)$. Then, we have, for $n+\d>0$, 
	\begin{multline*}
	a(n) =  2\pi i^{\k}(n+\delta)^{\frac{\kappa-1}{2}} \sum_{k\ge1} \dfrac{1}{k} \sum_{\substack{0 \leq h < k \\ \gcd(h,k) = 1}} e^{- \frac{2\pi i (n+\delta) h}{k}}
	\\
	\times \sum_{r+\delta_{h,k} < 0} a_{h,k}(r) e^{ \frac{2\pi i}{k c_{k}}\left(r + \delta_{h,k}\right)[-h]_k } \left( \dfrac{\lvert r +\delta_{h,k} \rvert }{c_{k}} \right)^{ \frac{1 - \kappa}{2}} I_{-\kappa+1}\left( \dfrac{4\pi}{k} \sqrt{\dfrac{\lvert r +\delta_{h,k} \rvert(n + \delta)}{c_{k}}} \right).
	\end{multline*}
\end{theorem}

\subsection{The Jacobi theta function}\hspace{0 cm}
Recall the {\it Jacobi theta function} 
\begin{equation*}
\vartheta(w;\tau) := \sum_{n\in\mathbb{Z}+\frac12} q^{\frac{n^2}2} e^{2\pi in\left(w+\frac12\right)}\ \ \left(w\in \mathbb{C},\ q:=e^{2\pi i\tau}\right).
\end{equation*}
Let $z\in \mathbb{C}$ with $\text{Re}(z)>0$, $k\in\N$, $h\in\mathbb{Z}$ with $\gcd(h,k)=1$. Define $\omega_{h,k}$ by\footnote{For the explicit shape of $\omega_{h,k}$, see \cite[(5.2.4)]{AndrewsPartitions}. We however do not require it in this paper.}
\begin{equation*}
\eta\left(\frac1k(h+iz)\right) = e^{\frac{\pi i}{12k}(h-[-h]_k)} \omega_{h,k}^{-1} z^{-\frac12} \eta\left(\frac1k\left([-h]_k+\frac iz\right)\right).
\end{equation*}

Moreover, we require the following properties of $\vartheta(w;\t)$ from \cite[Lemma 2.1 (1), (2), (3)]{BHK}, \cite[Proposition 1.3 (7), (8)]{Zwegers}, and \cite{EZ}.

\begin{lemma}\label{lem:thetaprops}
Suppose that $w\in \mathbb{C}$ and $\t\in \mathbb{H}$.
	\noindent
	\noindent
	\begin{enumerate}
		\item
		Suppose that $z\in\mathbb{C}$ with $\textnormal{Re}(z)>0$, $k\in\N$, and $h\in\mathbb{Z}$ with $\gcd(h,k)=1$. Then we have 
		\begin{equation*}
		\hspace{0.6 cm}\vartheta\!\left(w;\frac1k(h+iz)\right) = e^{\frac{\pi i}{4k}\left(h-[-h]_k\right)} e^{-\frac{3\pi i}{4}} \omega_{h,k}^{-3} \sqrt{\frac iz} e^{-\frac{\pi  kw^2}z} \vartheta\left(\frac{iw}z;\frac1k\left([-h]_k+\frac iz\right)\right).
		\end{equation*}
		\item
		Suppose that $\lambda,\mu\in\mathbb{Z}$. Then, with $\z:=e^{2\pi i w}$, we have 
		\begin{equation*}
		\vartheta(w+\lambda\tau+\mu;\tau) = (-1)^{\lambda+\mu} q^{-\frac{\lambda^2}2} \zeta^{-\lambda} \vartheta(w;\tau).
		\end{equation*}
		\item  We have
		\[	\vartheta(w;\tau)=-iq^{\frac{1}{8}} e^{-\pi i w}\prod_{n\geq 1}\left(1-q^n\right)\left(1-e^{2\pi i w}q^{n-1}\right)\left(1-e^{-2\pi i w}q^{n}\right).\]
		\item We have
		\[
		\vartheta\left(\frac{w}{\t};-\frac 1\t\right)=-i\sqrt{-i\t}e^{\frac{\pi iw^2}{\t}}\vartheta\left(w;\t\right).
		\]
		\item For $\l\in \mathbb{Q}$, $q^{\frac{\l^2}{2}}\vartheta(\l \t;\t)$ is a modular form of weight $\frac 12$ (on some congruence subgroup).
	\end{enumerate}
\end{lemma}

Moreover, using \cite[Lemma 2.1 (3)]{BHK}, \Cref{lem:thetaprops} (3),\,(4), we have the following lemma.

\begin{lemma}\label{newlem}
	Suppose that $\t\in \mathbb{H}$ and $a,b\in \mathbb{R}$. 
	\noindent
	\noindent
	\begin{enumerate}
		\item If $0<a<1$, then we have, as $\t\to i\infty$,
		\[
		\vartheta(a\tau+b;\tau) = -ie^{-\pi ib} q^{\frac18-\frac a2} \left(1+O\left(q^{\min\{a,1-a\}}\right)\right).\]
		\item  For $b\notin\mathbb{Z}$, we have, as $\t\to i\i$,
		\[
		\vartheta\left(b;\t\right)\sim -2\sin\left(\pi b\right)e^{\frac{\pi i\t}{4}}.
		\]
	\item  If $0<b<1$, then we have, as $z\to 0$, 
	\[
	\vartheta\left(aiz+b;iz\right)\sim -\frac{e^{-2\pi iab+\pi ia}e^{\left(-b^2-\frac 14+b\right)\frac{\pi}{z}}}{\sqrt{z}}.
	\]
	\item For $a\notin\mathbb{Z}$, we have, as $z\to 0$, 
	\[
	\vartheta\left(a i z;iz\right)\sim -\frac{2i\sin(\pi a)}{\sqrt{z}}e^{-\frac{\pi}{4z}}.
	\]
	\end{enumerate}
\end{lemma}

We also require the following asymptotics of the $I$-Bessel function.
\begin{lemma}\label{asymp}
	We have, as $x\to \infty$,
	\[
	I_1(x)\sim \frac{e^x}{\sqrt{2\pi x}}.
	\]	
\end{lemma}

\subsection{A Tauberian theorem} We require a modified version of Ingham's Tauberian Theorem \cite{I}, stated in \cite[Theorem 1.1]{BJM}. For $\Delta\geq0$, define
\begin{equation*}
R_\Delta := \{x+iy:\, x,y\in\R,\, x>0,\, |y|\le\Delta x\}.
\end{equation*}	
\begin{proposition}\label{Ingham}
	Let $B(q)=\sum_{n\ge0}b(n)q^n$ be a power series whose radius of convergence is at least one and assume that the $b(n)$ are non-negative and weakly increasing.
	Also suppose that $\nu$, $\b$, $\g\in\R$ with $\g>0$ exist such that\,\footnote{Throughout we use the principal branch to define the logarithm and, consequently, the square root.}
	\begin{equation}\label{Eqn1}
	B\left(e^{-t}\right) \sim \nu t^\b e^\frac\g t \quad\text{as } t \to 0^+,\qquad B\left(e^{-z}\right) \ll |z|^\b e^\frac{\g}{|z|} \quad\text{as } z \to 0,
	\end{equation}
	with the latter condition holding in each region\footnote{Throughout, we consider $z\to 0$ in $R_{\Delta}$ with $\Delta\ge 0$.} $R_\Delta$
	for $\Delta \geq 0$. Then we have
	\[
	b(n) \sim \frac{\nu\g^{\frac\b2+\frac14}}{2\sqrt\pi n^{\frac\b2+\frac34}}e^{2\sqrt{\g n}} \qquad\text{as}\  n \to \infty.
	\]
\end{proposition}

Moreover, we require the asymptotic behavior of $(q)_\infty$ as $q\to 1$.
\begin{lemma}\label{L:q-Pochhammer asymptotics}
	We have, with $q=e^{-z}$,
	\begin{align*}
	\left( q \right)_\infty &\sim \sqrt{\dfrac{2\pi}{z}} e^{- \frac{\pi^2}{6z}}\quad (\text{as}\ z\to 0).
	\end{align*}
\end{lemma}

\section{Proof of \Cref{A} and \Cref{BenCor}}\label{sec2}

In this section, we consider the case that $m$ is odd. For $m\in \mathbb{N}_{> 3}$ and $\ell\in \mathbb{N}$, we study the behavior of $G^{\ell}_m(\t)$, where 
\[
G_m(\t):=\frac{\vartheta\left(\t;m\t\right)}{\vartheta\left(2\t;m\t\right)}=q^{\frac 12}F_m(q),
\]
as $\tau\to \frac hk\in \mathbb{Q}$. For $a\in \mathbb{R}$, we write $a=\lfloor a\rfloor+\{a\}$. We have the following lemma (for general $m$). For $m\in \mathbb{N}_{>3}$, $\ell,k\in \mathbb{N}$, $d:=\gcd(m,k)$ with $\mu:=\frac md$, and $\kappa:=\frac kd$, define   
\begin{equation*}
\mathcal{A}_{m,\ell}(h,k):=\frac{\pi \ell}{\mu\kappa}\left(\left\{\frac{\mu h}{m}\right\}-\left\{\frac{\mu h}m\right\}^2-\left\{\frac{2\mu h}m\right\}+\left\{\frac{2\mu h}m\right\}^2\right).
\end{equation*}

\begin{lemma}\label{Fin1}
	Let $m\in \mathbb{N}_{>3}$ and $\ell,k\in \mathbb{N}$. For $0\le h<k$ with $\gcd(h,k)=1$, we have, as $z\to 0$, 
	\begin{align*}
G^{\ell}_m\left(\frac{h}{k}+\frac{iz}{k}\right)&=e^{\frac{\mathcal{A}_{m,\ell}(h,k)}{z}+O_{h,k}(1)}.
	\end{align*}
\end{lemma}
\begin{proof}
	It is enough to prove the claim for $\ell=1$. Indeed, the general case follows by taking the $\ell$-th power. Define
	\begin{equation}\label{theta}
	g_m(\t):=G_m\left(\frac{\t}{m}\right)=\frac{\vartheta\left(\frac{\t}m;\t\right)}{\vartheta\left(\frac{2\t}{m};\t\right)}.
	\end{equation}
 Then
	\begin{align}\label{c1}
	\text{Log}\left(g_m\left(\frac 1k\left(h+iz\right)\right)\right)
	&\equiv 
	\sum_{j\in \{0,1\}}(-1)^j \text{Log}\left(\vartheta\left(\frac{j+1}{mk}\left(h+iz\right);\frac 1k\left(h+iz\right)\right)\right),
	\end{align}
where, here and throughout, we use $\equiv$ to denote that the identity holds modulo $2\pi i\mathbb{Z}$. Now by \Cref{lem:thetaprops} (1) with $w=\frac{j+1}{mk}(h+iz)$ ($j\in \{0,1\}$), we have, as $z\to 0$, 
	\begin{align*}
	&\text{Log}\left(\vartheta\left(\frac{j+1}{mk}\left(h+iz\right);\frac 1k\left(h+iz\right)\right)\right)\\
	&\hspace{0.5 cm}\equiv-\frac{\text{Log}(z)}{2}-\frac{\pi\left(j+1\right)^2h^2}{m^2kz}
	+\text{Log}\left(\vartheta\left(\frac{i(j+1)h}{mkz}-\frac{j+1}{mk};\frac 1k\left(\left[-h\right]_k+\frac iz\right)\right)\right)+O_{h,k}(1).
	\end{align*}
	Plugging this into \eqref{c1}, we obtain, as $z\to 0$,
	\begin{align}\label{c2}
	&\text{Log}\left(g_m\left(\frac 1k\left(h+iz\right)\right)\right)\nonumber\\
	&\hspace{2 cm}\equiv\frac{3\pi h^2}{m^2kz}
	+\sum_{j\in\{0,1\}}(-1)^j\text{Log}\left(\vartheta\left(\frac{i(j+1)h}{mkz}-\frac{j+1}{mk};\frac 1k\left(\left[-h\right]_k+\frac iz\right)\right)\right)
	+O_{h,k}(1)\nonumber\\
	&\hspace{2 cm}=\frac{3\pi h^2}{m^2kz}\!+\!\!\!\sum_{j\in\{0,1\}}\!\!(-1)^j\text{Log}\left(\vartheta\left((j+1)\left(\frac{h}{m}\t_1+\a_1\right);\t_1 \right) \right)+O_{h,k}(1),
	\end{align}
	where 
	\[
	\a_1:=-\frac{h\left[-h\right]_k+1}{mk}\quad \text{and}\quad \t_1:=\frac 1k\left(\left[-h\right]_k+\frac iz\right).
	\]
	By \Cref{lem:thetaprops} (2) with $\l=\lfloor a\rfloor,\ \mu=0,\ w=\{a\}\t+b$, we have 
	\begin{align*}
	\vartheta\left(a\t+b;\t\right)
	&=(-1)^{\lfloor a\rfloor}e^{-2\pi i\lfloor a\rfloor b} q^{-\frac{\lfloor a\rfloor^2}{2}-\lfloor a\rfloor \{a\}}\vartheta\left(\{a\}\hspace{-0.01 cm}\t+b;\t\right).
	\end{align*}
	Applying this with $a=\frac hm$, $b=\a_1$, and $\t= \t_1$, we obtain, 
	\begin{align}\label{c3}
	&\text{Log}\left(\vartheta\left(\left(\left\lfloor\frac hm\right\rfloor+\left\{\frac hm\right\}\right)\t_1+\a_1;\t_1\right)\right)\nonumber\\
	&\hspace{1 cm}\equiv-\!\pi i\left\lfloor\frac hm\right\rfloor^2\t_1\!-\!2\pi i\left\lfloor\frac hm\right\rfloor\left\{\frac hm\right\}\t_1+\!\text{Log}\left(\vartheta\left(\left\{\frac hm\right\}\t_1\!+\!\a_1;\t_1\right)\right)
	+O_{h,k}(1).
	\end{align}
	Now, by \Cref{newlem} (1),\,(2) with $a=\{\frac hm\}$, $b=\a_1$, and $\t= \t_1$, we have, 
	\[
	\text{Log}\left(\vartheta\left(\left\{\frac hm\right\}\t_1+\a_1;\t_1\right)\right)\equiv\frac{\pi i}{4}\t_1-\pi i\left\{\frac hm\right\}\t_1+O_{h,k}(1)\ \ (\text{as}\ z\to 0).
	\]
	Plugging this into \eqref{c3}, we obtain, as $z\to 0$, 
	\begin{align}\label{c4}
	&\text{Log}\left(\vartheta\left(\left(\left\lfloor\frac hm\right\rfloor+\left\{\frac hm\right\}\right)\t_1+\a_1;\t_1\right)\right)\nonumber\\
	&\hspace{5.5 cm}
	\equiv\frac{\pi i}{4}\t_1+\pi \left(\frac{h^2}{m^2}+\left\{\frac hm\right\}-\left\{\frac hm\right\}^2\right)\frac{1}{kz}+O_{h,k}(1).
	\end{align}
	In the same way, we have, as $z\to 0$,
	\begin{align*}
	&\text{Log}\left(\vartheta\left(\left(\left\lfloor\frac{2h}m\right\rfloor+\left\{\frac{2h}m\right\}\right)\t_1+2\a_1;\t_1\right)\right)\\
	&\hspace{6 cm}
	\equiv\frac{\pi i}{4}\t_1+\pi \left(\frac{4h^2}{m^2}+\left\{\frac{2h}m\right\}-\left\{\frac{2h}m\right\}^2\right)\frac{1}{kz}
	+O_{h,k}(1).
	\end{align*}
	Plugging this and \eqref{c4} into \eqref{c2}, we obtain, as $z\to 0$, 
	\begin{align*}
	\text{Log}\left(g_m\left(\frac 1k\left(h+iz\right)\right)\right)&\equiv\frac{\pi }{kz}\left(\left\{\frac{h}m\right\}-\left\{\frac{h}m\right\}^2-\left\{\frac{2h}m\right\}+\left\{\frac{2h}m\right\}^2\right)
	+O_{h,k}(1).
	\end{align*}
		Exponentiating and using \eqref{theta} with $\t\mapsto m\t$, $m=d\mu$, and $k=d\kappa$, we conclude the proof.\qedhere   
\end{proof}

Next, we maximize $|G^{\ell}_m(\frac 1k(h+iz))|$. 

\begin{lemma}\label{lem1}
	For $m\in \mathbb{N}_{>3}$ odd and $\ell\in \mathbb{N}$, $|G^{\ell}_m(\frac 1k(h+iz))|$ is maximal, as $z\to 0$, for $\frac hk\in \{\frac{m-1}{2m},\frac{m+1}{2m}\}$. 
\end{lemma}
\begin{proof}
Again we may assume $\ell=1$. By \Cref{Fin1}, as $z\to 0$, $|G_m(\frac 1k(h+iz))|$ is maximal for the maximal value of 
	\begin{equation}\label{Max}
	\begin{cases}
	\frac{1}{\mu\kappa}\left(3\left\{\frac{\mu h}m\right\}^2-\left\{\frac{\mu h}m\right\}\right) &\text{if}\quad 0\le \left\{\frac{\mu h}m\right\}<\frac 12,\\[0.8 em]
	\frac{1}{\mu\kappa}\left(3\left\{\frac{\mu h}m\right\}^2-5\left\{\frac{\mu h}m\right\}+2\right) &\text{if}\quad \frac 12\le \left\{\frac{\mu h}m\right\}<1.
	\end{cases}
	\end{equation}
	Let $d\mid m$. If $d=1$, then $\mu=m$, $\kappa=k$, and $\{\frac{\mu h}{m}\}=\{h\}=0$. So, by \Cref{Fin1}, $G_m(\frac 1k(h+iz))=O_{h,k}(1)$, as $z\to 0$. We next assume $d>1$. Plugging in $\mu=\frac md$ and $\kappa=\frac kd$, by \eqref{Max}, we need to maximize (note that $d$ is odd since $m$ is odd) 
	\begin{align*}
	&\begin{cases}
	\frac{d^2}{mk}\left(3\left\{\frac{h}d\right\}^2-\left\{\frac{h}d\right\}\right) &\text{if}\quad 0\le \left\{\frac{h}d\right\}<\frac 12,\\[0.8 em]
	\frac{d^2}{mk}\left(3\left\{\frac{h}d\right\}^2-5\left\{\frac{h}d\right\}+2\right) &\text{if}\quad \frac 12\le \left\{\frac{h}d\right\}<1,
	\end{cases}\nonumber\\
	&\hspace{1 cm}=\begin{cases}
	\frac{d^2}{mk}\left(\frac{3j^2}{d^2}-\frac jd\right) &\text{if}\quad h\in d\mathbb{N}_0+j,\ 0\le j\le \left\lfloor\frac{d-1}{2}\right\rfloor,\\[0.5 em]
	\frac{d^2}{mk}\left(\frac{3j^2}{d^2}-\frac{5j}d+2\right) &\text{if}\quad h\in d\mathbb{N}_0+j,\ \left\lfloor\frac{d+1}{2}\right\rfloor\le j\le d-1.
	\end{cases}
	\end{align*}
	Since $m$ is odd and $d\mid m$, making the change of variables $j\mapsto d-j$ in the second case, we thus need to maximize 
	\begin{align}\label{ref1}
	\begin{cases}
	\frac{d^2}{mk}\left(\frac{3j^2}{d^2}-\frac jd\right) &\text{if}\quad h\in d\mathbb{N}_0+j,\ 0\le j\le \frac{d-1}{2},\\[0.5 em]
	\frac{d^2}{mk}\left(\frac{3j^2}{d^2}-\frac{j}d\right) &\text{if}\quad h\in d\mathbb{N}-j,\ 1\le j\le \frac{d-1}2.
	\end{cases}
	\end{align}
	Observe that for $d=3$ these expressions vanish. So for the remainder of the proof, we assume $d>3$. For $m$ and $d$ fixed the expressions in \eqref{ref1} are maximized for $j=\frac{d-1}{2}$. The corresponding contribution is $\frac{(d-1)(d-3)}{4mk}$. Moreover, one sees that this is maximal for $k$ minimal with $\gcd(m,k)=d$, which occurs for $k=d$.
	\rm
	The expression we need to maximize in this case is $\frac{(d-1)(d-3)}{4md}$.  As $d\mid m$, this is maximized for $d=m$. The corresponding values of $h$ are $\frac{m-1}{2}$ and $\frac{m+1}{2}$ which are clearly coprime to $k=m$.\qedhere 
\end{proof}

\begin{lemma}\label{lem2}
	For $m\in \mathbb{N}_{>3}$ odd, $\ell\in \mathbb{N}$, and $h\in \{\frac{m-1}{2},\frac{m+1}{2}\}$, we have, as $z\to 0$,
	\[
	F^{\ell}_m\left(e^{\frac{2\pi i}{m}\left(h+iz\right)}\right)\sim C^{\ell}_{m,h}e^{\left(\frac 14-\frac 1m+\frac{3}{4m^2}\right)\frac{\pi\ell}z}
	\]
	with 
	\[
	C_{m,h}:=\begin{cases}
	-ie^{\frac{5\pi i}{2m}-\frac{3\pi i}{m^2}}&\quad \text{if}\ h=\frac{m-1}{2},\\ 
	ie^{-\frac{5\pi i}{2m}+\frac{3\pi i}{m^2}}&\quad \text{if}\ h=\frac{m+1}{2}. 
	\end{cases}
	\]
\end{lemma}
\begin{proof}
	Note that it is enough to prove the case $\ell=1$. By \Cref{lem:thetaprops} (3), for $m\in \mathbb{N}_{>3}$ odd, we have, as $z\to 0$,
	\begin{align*}
	F_m\left(e^{\frac{2\pi i}{m}\left(\frac{m-1}2+iz\right)}\right)&\sim -i e^{\frac{\pi i}{2m}}\frac{\vartheta\left(\frac 12-\frac{1}{2m}+\frac{iz}m;iz\right)}{\vartheta\left(1-\frac{1}{m}+\frac{2iz}m;iz\right)}.	\end{align*}
	Moreover, by definition, $c^{[1]}_{m}(n)\in \mathbb{R}$. So, $F_m(\overline{q})=\overline{F_m(q)}$, and thus we only need to prove the claim for $\frac{m-1}{2}$. \Cref{newlem} (3) finishes the proof.\qedhere
\end{proof}

\begin{proof}[Proof of \Cref{A}]
	By \Cref{lem:thetaprops} (5), $q^{-\frac{3\ell}{2m}+\frac{\ell}2}F^{\ell}_m(q)=q^{-\frac{3 \ell}{2m}}G_m^\ell(\tau)$ is a modular function. By \Cref{lem1} its maximal growth occurs at the cusps $\frac{m-1}{2m}$ and $\frac{m+1}{2m}$. Using \Cref{asymp}, \Cref{lem2}, and \eqref{phase}, \Cref{Thm:Zuckerman} now gives, as $n\to \infty$, 
	\[
	c^{[\ell]}_m(n)\sim\left(\frac{\left(1-\frac 1m\right)\left(1-\frac 3m\right)\ell}{2m}\right)^{\frac 14} \frac{(-1)^n\cos\left(\pi\left(\phi_m\ell -\frac{n}{m}\right)\right)e^{\pi\sqrt{\frac{2\left(1-\frac 1m\right)\left(1-\frac 3m\right)\ell }{m}n }}}{n^{\frac 34}}.
	\]
	Here we require that the right-hand side does not vanish, which we now prove. Precisely, we need to show that, for $n\in \mathbb{N}$ and $\ell\in \mathbb{N}$,
	\[
	\cos\left(\pi\left(\phi_m\ell -\frac{n}{m}\right)\right)\neq 0.
	\]
	For this, it suffices to prove that $\phi_m\ell -\frac{n}{m} \notin \mathbb{Z}+\frac{1}{2}$. Suppose for contradiction that, for some $j\in \mathbb{Z}$, 
	\begin{align*}
	\phi_m\ell -\frac{n}{m}=j+\frac 12\Rightarrow (m-2)(m-3)\ell -2mn=2m^2j+m^2,
	\end{align*}
	which cannot hold because the left-hand side is even whereas the right-hand side is odd, as $m$ is odd. This concludes the proof.
\end{proof}

\begin{proof}[Proof of Corollary~\ref{BenCor}]
	By \Cref{A}, for $n\gg 1$ with $n\equiv r\pmod{2m}$,
	\begin{align}\label{eqn1}
		\textnormal{sgn}\left(c^{[\ell]}_{m}(n)\right)
	&=
		(-1)^r\textnormal{sgn}\left(
		\cos\left(\pi\left(\phi_m\ell-\frac{r}{m}\right)\right)
		\right).
	\end{align}

	Thus, it suffices to determine $\textnormal{sgn}(\cos(\pi( \phi_m\ell-\frac{ r}{m})))$.
	We only consider the case $\ell$ even, the case $\ell$ odd is shown analogously. Using \eqref{phase}, we write
	\begin{align}\label{eqn2}
	\pi\left(\phi_m\ell-\frac{r}{m}\right)
	&=\pi\left(\frac{\ell}{2}+\frac{3\ell}{m^2} -\frac{ R}{m}\right),
	\end{align}
	where 
	\begin{equation}\label{eqn3}
	R:=R_{r,\ell}:=r+\frac{5\ell}{2}.
	\end{equation}
	Using \eqref{eqn2}, we obtain 
	\begin{align}\label{eqn4}
	\cos\left(\pi\left(\phi_m\ell-\frac{r}{m}\right)\right)
	=(-1)^{\frac{\ell}{2}}\cos\left(\pi\left( \frac{ R}{m}-\frac{3\ell}{m^2}\right)\right).
	\end{align}
	Now note that by \eqref{eqn3}, we have 
	\[5\le R<\frac{5m}{2}.\]
	
	Next, we split the range of $R$. For
	$5\le R\le \frac{m-1}{2}$, we have  
	\[
	0<\pi\left(\frac{ R}{m}-\frac{3\ell}{m^2}\right)<\frac{\pi}{2}.
	\]
	Thus
	\begin{equation}\label{eqn5}
	\textnormal{sgn}\left(\cos\left(\pi\left( \frac{ R}{m}-\frac{3\ell}{m^2}\right)\right)\right)=1\ \ \text{for}\ \ 5\le R\leq \frac{m-1}{2}.
	\end{equation}

	Next, for $\frac{m+1}{2}\le R\le \frac{3m-1}{2}$, we have 
	\[
	\frac{\pi}{2}< \pi\left(\frac{ R}{m}-\frac{3\ell}{m^2}\right)<\frac{3\pi}{2}.
	\]
	Thus, we obtain 
	\begin{equation}\label{eqn6}
	\textnormal{sgn}\left(\cos\left(\pi\left( \frac{ R}{m}-\frac{3\ell}{m^2}\right)\right)\right)=-1\ \ \text{for}\ \ \frac{m+1}{2}\le R\le \frac{3m-1}{2}.
	\end{equation}
	
	Finally, for the remaining range of $R$, i.e., $\frac{3m+1}{2}\le R< \frac{5m}{2}$, we have 
	\[
	\frac{3\pi}{2}< \pi\left(\frac{ R}{m}-\frac{3\ell}{m^2}\right)<\frac{5\pi}{2}.
	\]
	Thus, we obtain
	\begin{equation*}
	\textnormal{sgn}\left(\cos\left(\pi\left( \frac{ R}{m}-\frac{3\ell}{m^2}\right)\right)\right)=1\ \ \text{for}\ \ \frac{3m+1}{2}\le R<\frac{5m}{2}.
	\end{equation*}
	Combining this, \eqref{eqn5}, \eqref{eqn6}, \eqref{eqn4}, \eqref{eqn3}, and \eqref{eqn1}, we obtain the claim.\qedhere 
\end{proof}

\section{Proof of \Cref{mainthmeven} and \Cref{mainthmevencor}}\label{sec3}

In this section, we consider the case that $m$ is even. We first determine the asymptotics of $F^{\ell}_m(-e^{-z})$ as $z\to 0$. Note that, by the fact $m$ even and \eqref{Eqn0}, we have
\begin{equation}\label{Eqn2}
F^{\ell}_m(-q)=\frac{\left(-q,-q^{m-1};q^m\right)^{\ell}_{\infty}}{\left(q^2,q^{m-2};q^m\right)^{\ell}_{\infty}}=\sum_{n\ge 0}(-1)^n c^{[\ell]}_m(n)q^n.
\end{equation}
We set
\begin{equation}\label{def}
d^{[\ell]}_m(n):=(-1)^n c^{[\ell]}_m(n).
\end{equation}

\begin{lemma}\label{Lem1}
	For $n\in \mathbb{N}_0$, $d^{[\ell]}_m(n)\ge 0$ and $\{d^{[\ell]}_m(n)\}_{n\ge 0}$ is weakly increasing.
	
\end{lemma}
\begin{proof}
	From \eqref{Eqn2}, it directly follows that $d^{[\ell]}_m(n)\ge 0$ for $n\in \mathbb{N}_0$. 
	
	To prove that $\{d^{[\ell]}_m(n)\}_{n\ge 0}$ is weakly increasing, it suffices to show that $(1-q)F^{\ell}_m(-q)$ has non-negative coefficients. By \eqref{Eqn2}, we have 
	\begin{align*}
	\left(1-q\right)F^{\ell}_m(-q)
	&=\frac{\left(-q^{m+1},-q^{m-1};q^m\right)^{\ell}_{\infty}}{\left(1-q\right)^{\ell-1}\left(q^{m+2},q^{m-2};q^m\right)^{\ell}_{\infty}}
	\end{align*}
	which has non-negative coefficients. This concludes the proof.\qedhere     
\end{proof}

Next, we determine the asymptotics of $F^{\ell}_m(-e^{-z})$, as $z\to 0$.

\begin{lemma}\label{Lem2}
	We have, as $z\to 0$,
	\[
	F^{\ell}_m\left(-e^{-z}\right)\sim \frac{e^{\frac{\pi^2\ell}{2mz}}}{2^{\ell}\sin^{\ell}\left(\frac{2\pi}{m}\right)}.
	\]
\end{lemma}
\begin{proof}
	Note that it is enough to prove the case $\ell=1$. By \eqref{Eqn2} and \Cref{lem:thetaprops} (3), we have  
	\begin{align*}
	F_m(-q)=-\frac{ie^{-\pi i\t}\left(q^m;q^m\right)^2_{\infty}\vartheta\left(2\t;2m\t\right)}{\left(q^{2m};q^{2m}\right)_{\infty}\vartheta\left(\t;m\t\right)\vartheta\left(2\t;m\t\right)}.
	\end{align*}
	Using this with $\t=\frac{iz}{2\pi}$, \Cref{L:q-Pochhammer asymptotics} and \Cref{newlem} (4) imply the claim.\qedhere  
\end{proof}

\begin{proof}[Proof of \Cref{mainthmeven}]
	By \Cref{Lem1}, $d^{[\ell]}_m(n)$ satisfies the monotonicity and non-negativity required to apply \Cref{Ingham}. Moreover, \Cref{Lem2} shows that $F^{\ell}_m(-e^{-z})$ satisfies \eqref{Eqn1} with $\nu=\frac{1}{2^{\ell}\sin^{\ell}(\frac{2\pi}{m})}, \b=0,$ and $\g=\frac{\pi^2\ell}{2m}$. By \Cref{Ingham}, we have, as $n\to \infty$, 
	\[
	d^{[\ell]}_m(n)\sim \frac{\ell^{\frac 14}}{2^{\ell+\frac 54}m^{\frac 14}\sin^{\ell}\left(\frac{2\pi}{m}\right)}\frac{e^{\pi\sqrt{\frac{2\ell }{m}n}}}{n^{\frac 34}}.
	\]
	By \eqref{def}, we conclude the proof.\qedhere 
\end{proof}

\begin{proof}[Proof of \Cref{mainthmevencor}]
	The claim follows from \Cref{Lem1} and \eqref{def}.\qedhere 
\end{proof}

\section{Proof of \Cref{A1} and \Cref{FlipCor}}\label{sec4}
 For $m\in \mathbb{N}_{> 3}$ and $\l \in \mathbb{N}$, we study the main growth of $\mathcal{G}^{\l}_m(\t)$, where
\[
\mathcal{G}_m(\t):=\frac{\vartheta\left(2\t;m\t\right)}{\vartheta\left(\t;m\t\right)}=q^{-\frac 12}\mathcal{F}_m(q),
\]
as $\tau\to \frac hk\in \mathbb{Q}$. Noting that $\mathcal{G}_m=G_m^{-1}$, \Cref{Fin1} implies 
\rm
the following: 
\begin{lemma}\label{Fin1a}
	Let $m\in \mathbb{N}_{>3}$ and $\l, k\in \mathbb{N}$. For $0\le h<k$ with $\gcd(h,k)=1$, $d:=\gcd(m,k)$ with $\mu:=\frac md$ and $\kappa:=\frac kd$, we have,
	\begin{align*}
\mathcal{G}^{\l}_m\left(\frac{h}{k}+\frac{iz}{k}\right)&=e^{-\frac{\mathcal{A}_{m,\l}(h,k)}{z}+O_{h,k}(1)}\ \ (\text{as}\ z\to 0).
	\end{align*}
\end{lemma}

Next, we maximize $|\mathcal{G}^{\l}_m(\frac 1k(h+iz))|$.

\begin{lemma}\label{lem1a}
	For $m\in \mathbb{N}_{>3}$ with $m\equiv \pm 1 \pmod{6}$ and $\l\in \mathbb{N}$, $|\mathcal{G}^{\l}_m(\frac 1k(h+iz))|$ is maximal, as $z\to 0$, for $\frac hk\in \{\frac{m\mp 1}{6m}, \frac{5m\pm 1}{6m}\}$.
\end{lemma}
\begin{proof}
	By taking the negative of \eqref{Max} we see that, as $z\to 0$, $|\mathcal{G}^{\l}_m(\frac 1k(h+iz))|$ is maximal for the maximal value of
	\begin{equation}\label{Maxa}
	\begin{cases}
	\frac{1}{\mu\kappa}\left(-3\left\{\frac{\mu h}m\right\}^2+\left\{\frac{\mu h}m\right\}\right) &\text{if}\quad 0\le \left\{\frac{\mu h}m\right\}<\frac 12,\\[0.8 em]
	\frac{1}{\mu\kappa}\left(-3\left\{\frac{\mu h}m\right\}^2+5\left\{\frac{\mu h}m\right\}-2\right) &\text{if}\quad \frac 12\le \left\{\frac{\mu h}m\right\}<1.
	\end{cases}
	\end{equation}
	Let $d\mid m$. If $d=1$, then $\mu=m$, $\kappa=k$, and $\{\frac{\mu h}{m}\}=\{h\}=0$. So, by \Cref{Fin1a}, $\mathcal{G}^{\l}_m(\frac 1k(h+iz))=O_{h,k}(1)$ as $z\to 0$. Next, we assume $d>1$.
	Plugging in $\mu=\frac md$ and $\kappa=\frac kd$, by \eqref{Maxa}, making the change of variables $j\mapsto d-j$ in the second case, thus we need to maximize 
	\begin{align}\label{new1}
	\begin{cases}
	\frac{d^2}{mk}\left(-\frac{3j^2}{d^2}+\frac jd\right) &\text{if}\quad h\in d\mathbb{N}_0+j,\ 0\le j\le \frac{d-1}{2},\\[0.5 em]
	\frac{d^2}{mk}\left(-\frac{3j^2}{d^2}+\frac{j}d\right) &\text{if}\quad h\in d\mathbb{N}-j,\ 1\le j\le \frac{d-1}2.
	\end{cases}
	\end{align}
	We note that for $m$ and $d$ fixed these expressions are maximized if $|6j-d|$ is minimal. Now $|6j-d|$ is minimal for $j=\frac{d\mp 1}{6}$. In this case, \eqref{new1} is maximal for $k=d$. The expression we need to maximize in this case is $\frac{d-\frac 1d}{12m}$. This is maximized for $d=m$. The maximum value is $\frac{1}{12}(1-\frac{1}{m^2})$, which is positive for $m\in \mathbb{N}_{>3}$. Since $j=\frac{m\mp 1}{6}$, we have $m=6j\pm 1$ and hence \[\gcd\left(m, j\right)=\gcd\left(6j\pm 1, j\right)=1.\]
	Thus, for $h\in m\mathbb{N}_0+j$ (for the plus sign) and $k=m$ (resp. $h\in m\mathbb{N}-j$ and $k=m$ for the minus sign), we have $\gcd(h,k)=1$. Taking the representative $0\le h<m$ gives $h=j=\frac{m\mp 1}{6}$ and $h=m-j=\frac{5m\pm 1}{6}$.
\end{proof}

Next, we determine the growth of $\mathcal{F}^{\l}_m(e^{\frac{2\pi i}{m}(h+iz)})$ as $z\to 0$.

\begin{lemma}\label{lem2a}
	For $m\in \mathbb{N}_{>3}$ with $m\equiv \pm 1 \pmod{6}$, $\l\in \mathbb{N}$, and $h\in \{\frac{m\mp 1}{6}, \frac{5m\pm 1}{6}\}$, we have,
	\[
	\mathcal{F}^{\l}_m\left(e^{\frac{2\pi i}{m}\left(h+iz\right)}\right)\sim \mathcal{C}^{\l}_{m,h}e^{ \left(\frac{1}{12}-\frac{1}{12m^2}\right)\frac{\pi\l}z}\quad \left(\text{as}\ z\to 0\right),
	\]
	where 
	\[
	\mathcal{C}_{m,h}:=\begin{cases}
	e^{\frac{\pi i}{6}\mp \frac{\pi i}{6m}\pm \frac{\pi i}{m^2}} &\quad \text{if}\ h=\frac{m\mp 1}{6},\\
	e^{-\frac{\pi i}{6}\pm \frac{\pi i}{6m}\mp \frac{\pi i}{m^2}}&\quad \text{if}\ h=\frac{5m\pm 1}{6}.	 
	\end{cases}
	\]
\end{lemma}
\begin{proof}
	Again it is enough to consider the case $\l=1$.	Moreover, by definition, $\g^{[1]}_{m}(n)\in \mathbb{R}$. So, $\mathcal{F}_m(\overline{q})=\overline{\mathcal{F}_m(q)}$ and thus we only need to prove the claim for $h=\frac{m\mp 1}{6}$. 
	Using \Cref{lem:thetaprops} (3), for $h =\frac{m\mp 1}{6}$, we have, as $z\to 0$, 
	\begin{align*}
	\mathcal{F}_m\left(e^{\frac{2\pi i}{m}\left(\frac{m\mp 1}6+iz\right)}\right) \sim e^{\frac{\pi i}{6}} e^{\mp\frac{\pi i}{6m}}\frac{\vartheta\left(\frac 13\mp\frac{1}{3m}+\frac{2iz}m;iz\right)}{\vartheta\left(\frac 16\mp\frac{1}{6m}+\frac{iz}m;iz\right)}.\end{align*}
	Now, \Cref{newlem} (3) finishes the proof.\qedhere  
\end{proof}

\begin{proof}[Proof of \Cref{A1}]
 By \Cref{lem:thetaprops} (5), $q^{\frac{3\l}{2m}-\frac{\l}2}\mathcal{F}^{\l}_m(q)$ is a modular function. By \Cref{lem1a} its maximal growth occurs at the cusps $\frac{m\mp1}{6m}$ and $\frac{5m\pm 1}{6m}$. Using \Cref{lem2a} and \Cref{asymp}, \Cref{Thm:Zuckerman} now gives, as $n\to \infty$, 
	\[
	\g^{[\l]}_m(n)\sim \left(\frac{\left(1-\frac 1m\right)\left(1+\frac 1m\right)\l}{6m}\right)^{\frac 14}\frac{T^{[\l]}_m(n)e^{\pi\sqrt{\frac{2\left(1-\frac 1m\right)\left(1+\frac 1m\right)\l}{3m}n }}}{n^{\frac 34}}.
	\]
	Here we require that the right-hand side does not vanish, which we now prove. Precisely, for $n,\l\in \mathbb{N}$ and $m\in \mathbb{N}_{>3}$ with $\gcd(m,6)=1$, we need to show
	\[
	T^{[\l]}_m(n)\neq 0.
	\]
	The proof is analogous to the argument given in \Cref{A}.\qedhere 
\end{proof}

\begin{proof}[Proof of \Cref{FlipCor}]
	By \Cref{A1}, for $n\gg1$, 
	\begin{equation*}
	\operatorname{sgn}\left(\g^{[\l]}_m(n)\right)
	=\operatorname{sgn}\left(T_m^{[\l]}(n)\right).
	\end{equation*}
	Since $n\equiv3r+2mj\pmod{6m}$,
	we obtain, using that $m\equiv \pm 1 \pmod{6}$, 
	\begin{align}
	T_m^{[\l]}(n)
	&=
	(-1)^r
	\cos\left(\pi
	\left(\frac{\l}{6}
	\pm \frac{1}{m}
	\left(
	r-\frac{\l}{6}+\frac{\l}{m}\right)
	\right)
	\right).
	\label{eq:flip-reduction}
	\end{align}
	We write $\l=6s+a$ with $0\le a\le 5$. Since $m\equiv\pm 1\pmod6$, we have
	\[
	R=R_a:=a\frac{\pm m-1}{6}+r-s\in \mathbb{Z}.
	\]
	Noting that 
	\begin{align*}
	\frac{\l}{6}\pm \frac 1m\left(r-\frac{\l}{6}+\frac{\l}{m}\right)
	=s\pm \frac 1m\left(R+\frac{\l}{m}\right),
	\end{align*}
	and using \eqref{eq:flip-reduction}, we have,
	\begin{align}
	T_m^{[\l]}(n)
	&=
	(-1)^{r+s}
	\cos\left(
	\pi
	\left(
	\frac{R}{m}+\frac{\l}{m^2}
	\right)
	\right).
	\label{eq:master}
	\end{align}

	Since $0<\l<\frac{m}{2}$ by assumption, we have
	\begin{equation*}
	0<\frac{\l}{m^2}<\frac{1}{2m}.
	\end{equation*}
	Thus, it is not hard to see that 
	\begin{equation}\label{eq:signlemma}
	\operatorname{sgn}\left(
	\cos\left(
	\pi
	\left(
	\frac{R}{m}+\frac{\l}{m^2}
	\right)
	\right)
	\right)
	=
	(-1)^{
		\chi_{R\!\pmod{2m}\in
			\left[\frac{m+1}{2},\,\frac{3m-1}{2}\right]}
	}.
	\end{equation}
	
	We now apply \eqref{eq:master} and \eqref{eq:signlemma}, distinguishing residue
	classes of $\l$ modulo $6$.
	
	\medskip
	\noindent
	\textbf{Case 1: $\l\equiv0\pmod6$.} 
	Here $a=0$ and $s=\frac{\l}{6}$. Thus $R=r-s$. We have 
	\[
	R\pmod{2m}\!\in\!
	\left[\frac{m+1}{2},\frac{3m-1}{2}\right]\Leftrightarrow\! r\pmod{2m}\!\in\!\left[
	\frac{m+1}{2}+\frac{\l}{6},\,
	\frac{3m-1}{2}+\frac{\l}{6}
	\right].
	\]
	It is not hard to show that this interval lies in $[0,2m-1]$. Thus we can skip the condition $r\!\pmod{2m}$. By \eqref{eq:master}, we conclude the claim in this case.

	\medskip
	\noindent
	\textbf{Case 2: $\l\equiv1\pmod6$.}  As the proof of this case is analogous to the case $\l\equiv 0\pmod{6}$, we skip it here.

	\medskip
	\noindent
	\textbf{Case 3: $\l\equiv2\pmod6$.} Again the proof follows that of the case $\l\equiv 0\pmod{6}$ and we therefore skip it here.

	\medskip
	\noindent
	\textbf{Case 4: $\l\equiv3\pmod6$.}
	Here $a=3$ and $s=\frac{\l-3}{6}$. Then $R=\frac{\pm m-1}{2}+r-s$. We have
	\begin{align*}
	&R\!\pmod{2m}\in
	\left[\frac{m+1}{2},\frac{3m-1}{2}\right]
	\Leftrightarrow r\!\pmod{2m}\!\in \left[\frac{m+1}{2}-\frac{\pm m-1}{2}+s, \frac{3m-1}{2}-\frac{\pm m-1}{2}+s\right].
	\end{align*}
	
	We first assume that $m \equiv 1 \pmod{6}$. Then our condition becomes
	\begin{align*}
	r\pmod{2m}\in
	\left[\frac{\l-3}{6}+1,m+\frac{\l-3}{6}\right].
	\end{align*}
	It is not hard to see that this again lies in $[0,2m-1]$. The claim follows in this case as $(\frac m3)=1$.

	Next, assume that $m\equiv -1\pmod 6$. Then our condition becomes
	\begin{align*}
	r\pmod{2m}
	&\in \left[m+s+1,2m+s\right]=\left[m+s+1,2m-1\right]\cup \left[2m,2m+s\right]\\
	&\hspace{0.4 cm}\equiv \left[m+s+1,2m-1\right]\cup\left[0,s\right]\pmod{2m},
	\end{align*}
	which is the complement of $[s+1,m+s]$ in $[0,2m-1]$. Thus
	\[
	(-1)^{\chi_{r\in \left[m+s+1,2m-1\right]\cup \left[0,s\right]}}=(-1)^{1+\chi_{r\in \left[s+1,m+s\right]}}.
	\]
	Noting that $(\frac m3)=-1$, the claim follows in this case.
	
	\medskip
	\noindent
	\textbf{Case 5: $\l\equiv4\pmod6$.}
	Here $a=4$ and $s=\frac{\l-4}{6}$. Then $R=\frac{2(\pm m-1)}{3}+r-s$. Thus we have 
	\begin{align*}
	&R\pmod{2m}\in
	\left[\frac{m+1}{2},\frac{3m-1}{2}\right]\\
	&\hspace{3 cm}
	\Leftrightarrow r\pmod{2m}\in \left[\frac{m+1}{2}-\frac{2\left(\pm m-1\right)}{3}+s, \frac{3m-1}{2}-\frac{2\left(\pm m-1\right)}{3}+s\right].
	\end{align*}
	If $m\equiv 1 \pmod{6}$, the above condition becomes 
	\begin{align*}
	r\pmod{2m}
	\in\left[-\frac m6+\frac{\l-4}{6}+\frac 76,\frac{5m}{6}+\frac{\l-4}{6}+\frac 16\right].
	\end{align*}
	Noting that the left endpoint is negative, we split the interval as 
	\begin{align*}
	&\left[-\frac m6+\frac{\l-4}{6}+\frac 76,-1\right]\cup \left[0,\frac{5m + 1}{6}+\frac{\l-4}{6}\right]\\
	&\hspace{5 cm}
	 \equiv \left[0,\frac{5m+1}{6}+\frac{\l-4}{6}\right]\cup \left[\frac{11m+7}{6}+\frac{\l-4}{6},2m-1\right]\pmod{2m}.
		\end{align*}
	This is the complement of $[\frac{5m+7}{6}+\frac{\l-4}{6}, \frac{11m+1}{6}+\frac{\l-4}{6}]$ in $[0,2m-1]$.
	The claim follows in this case.
	
	We next consider the case $m\equiv -1\pmod{6}$. Then our condition becomes 
	\begin{align}\label{l3}
	r\pmod{2m}
	\in\left[\frac{7m+7}{6}+\frac{\l-4}{6},\frac{13m+1}{6}+\frac{\l-4}{6}\right].
	\end{align}
	Note that $\frac{7m+7}{6}+\frac{\l-4}{6}<2m$ and $\frac{13m+1}{6}+\frac{\l-4}{6}>2m$. Thus we write the interval in \eqref{l3} as 
	\begin{align*}
	&\left[\frac{7m+7}{6}+\frac{\l-4}{6},2m-1\right]\cup \left[2m,\frac{13m+1}{6}+\frac{\l-4}{6}\right]\\
	&\hspace{5 cm}\equiv \left[0,\frac{m+1}{6}+\frac{\l-4}{6}\right]\cup \left[\frac{7m+7}{6}+\frac{\l-4}{6},2m-1\right]\pmod{2m}.
	\end{align*}
	This is the complement of $[\frac{m+7}{6}+\frac{\l-4}{6},\frac{7m+1}{6}+\frac{\l-4}{6}]$ in $[0,2m-1]$, yielding the claim in this case.
	
	\medskip
	\noindent
	\textbf{Case 6: $\l\equiv5\pmod6$.}
	Here $a=5$ and $s=\frac{\l-5}{6}$. Then $R=\frac{5(\pm m-1)}{6}+r-s$. We have
	\begin{align*}
	&R\pmod{2m}\in
	\left[\frac{m+1}{2},\frac{3m-1}{2}\right]\\
	&\hspace{4 cm}
	\Leftrightarrow r\pmod{2m}\in \left[\frac{m+1}{2}-\frac{5\left(\pm m-1\right)}{6}+s, \frac{3m-1}{2}-\frac{5\left(\pm m-1\right)}{6}+s\right].
	\end{align*}
	
	We first assume that $m\equiv 1 \pmod{6}$. Then our condition becomes 
	\begin{align*}
	r\pmod{2m}
	\in\left[\frac{4-m}{3}+\frac{\l-5}{6}, \frac{2m+1}{3}+\frac{\l-5}{6}\right].
	\end{align*}
	Noting that $\frac{4-m}{3}+\frac{\l-5}{6}<0$, we rewrite the above interval as 
	\begin{align*}
	&\left[\frac{4-m}{3}+\frac{\l-5}{6}, -1\right]\cup \left[0, \frac{2m+1}{3}+\frac{\l-5}{6}\right]\\
	&\hspace{5 cm}\equiv \left[0,\frac{2m+1}{3}+\frac{\l-5}{6}\right]\cup\left[\frac{5m+4}{3}+\frac{\l-5}{6},2m-1\right]\pmod{2m}.
	\end{align*}
	This is the complement of $[\frac{2m+4}{3}+\frac{\l-5}{6},\frac{5m+1}{3}+\frac{\l-5}{6}]$ in $[0,2m-1]$ and this claim follows in this case.
	
	We next assume $m\equiv -1\pmod{6}$. In this case, our condition becomes 
	\begin{align*}
	r\pmod{2m}
	\in \left[\frac{4m+4}{3}+\frac{\l-5}{6}, \frac{7m+1}{3}+\frac{\l-5}{6}\right].
	\end{align*}
	Noting that $\frac{4m+4}{3}+\frac{\l-5}{6}<2m$ and $\frac{7m+1}{3}+\frac{\l-5}{6}>2m$, we rewrite the above interval as 
	\begin{align*}
	&\left[\frac{4m+4}{3}+\frac{\l-5}{6}, 2m-1\right]\cup \left[2m, \frac{7m+1}{3}+\frac{\l-5}{6}\right]\\
	&\hspace{5 cm}\equiv \left[0,\frac{m+1}{3}+\frac{\l-5}{6}\right]\cup \left[\frac{4m+4}{3}+\frac{\l-5}{6}, 2m-1\right]\pmod{2m}.
	\end{align*}
	This is the complement of $[\frac{m+4}{3}+\frac{\l-5}{6},\frac{4m+1}{3}+\frac{\l-5}{6}]$ in $[0,2m-1]$ and the claim follows in this case.\qedhere 
\end{proof}

\section{Open questions}\label{sec5}

The results of this paper suggest several directions for future research.
\begin{enumerate}
	\item It would be interesting to extend Theorems \ref{A}, \ref{mainthmeven}, and \ref{A1}, by considering theta quotients of the form (for $a,b,m,\ell\in \mathbb{N}$ with $a\neq b$)
	\[
	\frac{\vartheta\left(a\t;m\t\right)^{\ell}}{\vartheta\left(b\t;m\t\right)^{\ell}}.
	\]
	For which choices of $a,b,m$, and $\ell$ do the Fourier coefficients have an eventual periodic sign pattern? Can this sign pattern be determined explicitly? More generally, can one characterize the dominant cusps in terms of $a,b,$ and $m$?
	\item In Theorem \ref{A1}, we assume $\gcd(m,6)=1$. Can this assumption be removed? More precisely, how do the dominant cusps and the resulting sign pattern change if $2\mid m$ or $3\mid m$?
	\item Our results raise a more general question: under what conditions do the Fourier coefficients of a modular theta quotient have an eventually periodic sign pattern?
\end{enumerate}

\section*{Declarations}
\subsection*{Data Availability}
Data sharing is not applicable to this article as no datasets were generated or analyzed during the current study.
\subsection*{Conflicts of interest} The authors declare that there are no conflicts of interest. 
\subsection*{Funding} Funding declarations can be found in the Acknowledgements Section.

\appendix
\section{Numerics}

In the following tables, we list the sign patterns of $c^{[\ell]}_m(n)$ and $\g^{[\ell]}_m(n)$ for $m\in \{5,7\}$, $1\le \ell \le 6$, and $70\le n\le m+69$. 
\begin{table}[H]
	\centering
	\renewcommand{\arraystretch}{1.7}
	\setlength{\tabcolsep}{4.2pt}

	\begin{tabular}{|@{\hspace{11pt}}c@{\hspace{11pt}}|c|c|c|c|c|}
		\hline
		$(m,\ell)$ & $(5,1)$ & $(5,2)$ & $(5,3)$ & $(5,4)$ & $(5,5)$\\
		\hline
		\rule[-2.1ex]{0pt}{5.2ex}$\textnormal{sgn}\left(c^{[\ell]}_m(n)\right)$ & $+-+--$ & $+-+--$ & $+-+-+$ & $+-+-+$ & $--+-+$\\
		\hline
		\rule[-2.1ex]{0pt}{5.2ex}$\textnormal{sgn}\left(\gamma^{[\ell]}_m(n)\right)$ & $++---$ & $+++--$ & $+++--$ & $-++--$ & $-+++-$\\
		\hline
	\end{tabular}\\[8.5pt] 
	
	\begin{tabular}{|@{\hspace{11pt}}c@{\hspace{11pt}}|c|c|c|c|}
		\hline
		$(m,\ell)$ & $(5,6)$ & $(7,1)$ & $(7,2)$ & $(7,3)$ \\
		\hline
		\rule[-2.1ex]{0pt}{5.2ex}$\textnormal{sgn}\left(c^{[\ell]}_m(n)\right)$ & $-++-+$ & $+-+-++-$ & $+-+-+-+$ & $--+-+-+$\\
		\hline
		\rule[-2.1ex]{0pt}{5.2ex}$\textnormal{sgn}\left(\gamma^{[\ell]}_m(n)\right)$ & $--++-$ & $+++---+$ & $+++----$ & $++++---$\\
		\hline
	\end{tabular}\\[8pt] 
	
	\begin{tabular}{|@{\hspace{11pt}}c@{\hspace{11pt}}|c|c|c|}
		\hline
		$(m,\ell)$ & $(7,4)$ & $(7,5)$ & $(7,6)$ \\
		\hline
		\rule[-2.1ex]{0pt}{5.2ex}$\textnormal{sgn}\left(c^{[\ell]}_m(n)\right)$  & $-+--+-+$ & $-+-++-+$ & $-+-+-++$\\
		\hline
		\rule[-2.1ex]{0pt}{5.2ex}$\textnormal{sgn}\left(\gamma^{[\ell]}_m(n)\right)$ & $-++++--$ & $--+++--$ & $--++++-$\\
		\hline
	\end{tabular}
	
	\vspace{0.2cm}
	\caption{Sign patterns for $c_m^{[\ell]}(n)$ and $\gamma_m^{[\ell]}(n)$.}
	\label{tab:horizontal_table}
\end{table}

\end{document}